\documentclass[12pt]{amsart}
\usepackage{amssymb}
\usepackage[all]{xy}

\newtheorem{theorem}{Theorem}[section]
\newtheorem{definition}[theorem]{Definition}
\newtheorem{proposition}[theorem]{Proposition}
\newtheorem{conjecture}[theorem]{Conjecture}

\begin{document}

\title[Maximal torus theory]{Maximal torus theory for compact quantum groups}

\author{Teodor Banica}
\address{T.B.: Department of Mathematics, Cergy-Pontoise University, 95000 Cergy-Pontoise, France. {\tt teo.banica@gmail.com}}

\author{Issan Patri}
\address{I.P.: Chennai Mathematical Institute, H1 Sipcot IT Park, Siruseri, Kelambakkam - 603103, India. {\tt ipatri@cmi.ac.in}}


\begin{abstract}
Associated to any compact quantum group $G\subset U_N^+$ is a canonical family of group dual subgroups $\widehat{\Gamma}_Q\subset G$, parametrized by unitaries $Q\in U_N$, playing the role of ``maximal tori'' for $G$. We present here a series of conjectures, relating the various algebraic and analytic properties of $G$ to those of the family $\{\widehat{\Gamma}_Q|Q\in U_N\}$.
\end{abstract}

\maketitle

\section*{Introduction}

We investigate here the notion of ``maximal torus'' for the compact quantum groups. In general, the maximal torus does not really exist. Given a closed subgroup $G\subset U_N^+$, what does exist, however, is a family of group dual subgroups $\widehat{\Gamma}_Q\subset G$, parametrized by unitaries $Q\in U_N$, which altogether play the role of the maximal torus. 

The construction, which goes back to \cite{bbd}, \cite{bv3}, is very simple, as follows:
$$C^*(\Gamma_Q)=C(G)\Big/\Big<(QuQ^*)_{ij}=0,\forall i\neq j\Big>$$

Here $u$ is the fundamental corepresentation of $G$, and the key observation is that the elements $g_i=(QuQ^*)_{ii}$ are group-like in the quotient algebra on the right. 
 
Based on growing evidence, coming from the recent quantum group literature, we will formulate here a series of conjectures, relating the various algebraic and analytic properties of $G$ to those of the family $\{\widehat{\Gamma}_Q|Q\in U_N\}$. Our main statements are as follows:
\begin{enumerate}
\item Character conjecture: Assuming that $G$ is connected, any nonzero element $P\in C(G)_{central}$ is left nonzero in one of the quotients $C^*(\Gamma_Q)$.

\item Amenability conjecture: The discrete quantum group $\widehat{G}$ is amenable if and only if all the discrete groups $\Gamma_Q$ are amenable.

\item Growth conjecture: The discrete quantum group $\widehat{G}$ has polynomial growth if and only if each $\Gamma_Q$ has polynomial growth.
\end{enumerate}

We believe all these conjectures to be true, non-trivial, and of course, of interest.

The paper is organized as follows: 1-2 are preliminary sections, in 3-4 we comment on the above conjectures, and in 5-6 we discuss Tannakian aspects.

\medskip

\noindent {\bf Acknowledgements.} T.B. would like to thank Minouille, wherever he is now, for spiritual guidance, and for everything. I.P. would like to thank Pierre Fima and Paris 7 University, for warm hospitality, at some early stage of this project.

\section{Maximal tori}

We use Woronowicz's compact quantum group formalism in \cite{wo1}, \cite{wo2}, with the extra axiom $S^2=id$. The precise definition that we will need is as follows:

\begin{definition}
If $A$ is a unital $C^*$-algebra, and $u\in M_N(A)$ is a unitary matrix, whose entries generate $A$, such that the following formulae define morphisms of $C^*$-algebras,
$$\Delta(u_{ij})=\sum_ku_{ik}\otimes u_{kj}\quad,\quad\varepsilon(u_{ij})=\delta_{ij}\quad,\quad S(u_{ij})=u_{ji}^*$$
we write $A=C(G)$, and call $G$ a compact matrix quantum group.
\end{definition}

The above maps $\Delta,\varepsilon,S$ are called comultiplication, counit and antipode. The basic examples include the compact Lie groups $G\subset U_N$, their $q$-deformations at $q=-1$, and the duals of the finitely generated discrete groups $\Gamma=<g_1,\ldots,g_N>$. See \cite{ntu}, \cite{wo1}.

There are many notions and constructions from the theory of compact groups which extend to the above setting. As an example here, having a closed quantum subgroup $H\subset G$ means to have a surjective morphism of $C^*$-algebras $\pi:C(G)\to C(H)$, mapping the standard generators of $C(G)$ to the standard generators of $C(H)$. See \cite{wo1}.

We have the following key construction, due to Wang \cite{wa1}:

\begin{definition}
The quantum groups $O_N^+,U_N^+$ constructed via
\begin{eqnarray*}
C(O_N^+)&=&C^*\left((u_{ij})_{i,j=1,\ldots,N}\Big|u=\bar{u},u^t=u^{-1}\right)\\
C(U_N^+)&=&C^*\left((u_{ij})_{i,j=1,\ldots,N}\Big|u^*=u^{-1},u^t=\bar{u}^{-1}\right)
\end{eqnarray*}
where $(\bar{u})_{ij}=u_{ij}^*$, $(u^t)_{ij}=u_{ji}$, $(u^*)_{ij}=u_{ji}^*$, are called free analogues of $O_N,U_N$.
\end{definition}

Here the existence of the morphisms $\Delta,\varepsilon,S$, as in Definition 1.1, comes from the universality property of the above algebras. Observe that by dividing by the commutator ideal, we obtain respectively the algebras $C(O_N),C(U_N)$. Thus, we have inclusions $O_N\subset O_N^+,U_N\subset U_N^+$. These inclusions are far from being isomorphisms. See \cite{wa1}.

With the definition of closed quantum subgroups given above, the quantum groups $G$ appearing as in Definition 1.1 are exactly the closed subgroups $G\subset U_N^+$. See \cite{wa1}.

The notion of diagonal subgroup goes back to \cite{bv3}. The idea is very simple:

\begin{definition}
Given a closed subgroup $G\subset U_N^+$, we set
$$C^*(\Gamma_1)=C(G)\Big/\Big<u_{ij}=0,\forall i\neq j\Big>$$ 
with $u\in M_N(C(G))$ being the fundamental corepresentation. 
\end{definition}

As explained in \cite{bv3}, the above quotient algebra is indeed cocommutative, and its generators $g_i=u_{ii}$ are group-like. Thus, $\Gamma_1=<g_1,\ldots,g_N>$ is a usual discrete group. Observe that $C^*(\Gamma_1)$ is not necessarily the full group algebra of $\Gamma_1$, but rather a certain quotient of it. For simplicity of presentation, we will use however the notation $C^*(\Gamma_1)$.

Observe that in the classical case, $G\subset U_N$, we obtain in this way the dual of the diagonal torus, $\Gamma_1=\widehat{G\cap\mathbb T^N}$, with $\mathbb T^N\subset U_N$ being the group of diagonal unitaries.

A key extension of the above construction, obtained by using a ``spinning'' matrix $Q\in U_N$, was proposed in \cite{bbd}. The idea is once again very simple, as follows:

\begin{definition}
Given a closed subgroup $G\subset U_N^+$, and a matrix $Q\in U_N$, we set
$$C^*(\Gamma_Q)=C(G)\Big/\Big<v_{ij}=0,\forall i\neq j\Big>$$ 
where $v=QuQ^*$, with $u\in M_N(C(G))$ being the fundamental corepresentation. 
\end{definition}

Observe that for the identity matrix $Q=1$ we obtain indeed the discrete group $\Gamma_1$ from Definition 1.3. In general, we do not really have a new notion here, because $\Gamma_Q$ is nothing but the group $\Gamma_1$ from Definition 1.3, taken for the quantum group $(G,v)$. 

The theoretical interest in this slight generalization of Definition 1.3 comes from the following fundamental result, due to Woronowicz \cite{wo1}:

\begin{theorem}
Any group dual subgroup $\widehat{\Lambda}\subset G$ must appear as 
$$\widehat{\Lambda}\subset\widehat{\Gamma}_Q\subset G$$
for a certain matrix $Q\in U_N$. 
\end{theorem}

\begin{proof}
As explained in \cite{wo1}, the finite dimensional unitary corepresentations of $\widehat{\Lambda}$ are completely reducible, with the irreducible corepresentations being all 1-dimensional, corresponding to the group elements $g\in\Lambda$. Thus, such a corepresentation must be of the form $v=QwQ^*$, with $w=diag(g_1,\ldots,g_N)$ and $Q\in U_N$. We conclude that the embeddings $\widehat{\Lambda}\subset U_N^+$ come from the quotient maps $C(U_N^+)\to C^*(\Lambda)$ of type $u\to QwQ^*$, and so the subgroups $\widehat{\Lambda}\subset G\subset U_N^+$ must appear as in the statement. See \cite{bbd}, \cite{wo1}.
\end{proof}

We have as well the following related result, from \cite{bv3}:

\begin{proposition}
If $g_1,\ldots,g_N\in\Gamma_Q$ are pairwise distinct, then $\widehat{\Gamma}_Q\subset G$ is a maximal group dual subgroup, i.e. there is no bigger group dual subgroup $\widehat{\Gamma}_Q\subset\widehat{\Lambda}\subset G$.
\end{proposition}

\begin{proof}
By rotating, we can assume $Q=1$. Given a subgroup $\widehat{\Gamma}_1\subset\widehat{\Lambda}\subset G$, let us denote by $w\leftarrow v\leftarrow u$ the corresponding fundamental corepresentations. We have then $v=Pdiag(h_1,\ldots,h_N)P^*$, for a certain $P\in U_N$, where $h_1,\ldots,h_N$ are the generators of $\Lambda$. In addition, the quotient map $\Lambda\to\Gamma_1$ must send $h_i\to g_{\sigma(i)}$, for a certain permutation $\sigma\in S_N$. We deduce that $w=diag(g_1,\ldots,g_N)$ commutes with $R=\sigma P$, which reads $R_{ij}(g_i-g_j)=0$, and so $R_{ij}=0$ for $i\neq j$. But this gives $\Lambda=\Gamma_1$. See \cite{bv3}.
\end{proof}

The considerations in \cite{bv3} were motivated by the root systems for half-classical quantum groups. The general problem here, still open, can benefit from the systematic approach to the half-liberation in \cite{bdu}. We will be back to these topics in section 5 below.

As for the considerations in \cite{bbd}, these were motivated by a key rigidity conjecture in noncommutative geometry \cite{gos}, solved in the meantime by Goswami and Joardar \cite{gjo}. The main observation in \cite{bbd} was the fact that a non-classical group dual cannot act on a compact connected Riemannian manifold. We will be back to this in section 6 below.

\section{Basic examples}

In this section we discuss, following \cite{bbd}, some basic examples of the construction $(G,Q)\to\Gamma_Q$ from Definition 1.4. In the classical case, the result is as follows:

\begin{proposition}
For a closed subgroup $G\subset U_N$ we have
$$\widehat{\Gamma}_Q=G\cap(Q^*\mathbb T^NQ)$$
where $\mathbb T^N\subset U_N$ is the group of diagonal unitary matrices.
\end{proposition}

\begin{proof}
This is indeed clear at $Q=1$, where $\Gamma_1$ appears by definition as the dual of the compact abelian group $G\cap\mathbb T^N$. In general, this follows by conjugating by $Q$.
\end{proof}

We denote by $F_N=<g_1,\ldots,g_N>$ the free group on $N$ generators. With this convention, here are the computations for Wang's quantum groups in \cite{wa1}:

\begin{proposition}
The construction $G\to\Gamma_Q$ is as follows:
\begin{enumerate}
\item For $G=U_N^+$ we obtain $\Gamma_Q=F_N$, for any $Q\in U_N$.

\item For $G=O_N^+$ we have $\Gamma_Q=F_N/<R_{ij}\neq0\implies g_ig_j=1>$, where $R=QQ^t$.
\end{enumerate}
\end{proposition}

\begin{proof}
These results are well-known, the proof being as follows:

(1) At $Q=1$ this is clear, and in the general case $Q\in U_N$ this follows from $\Gamma_1=F_N$, and from the fact that $C(U_N^+)$ is isomorphic to itself via $u\to QuQ^*$.

(2) At $Q=1$ this is clear, and we obtain the group $\Gamma_1=\mathbb Z_2^{*N}$. In general now, with $v=QuQ^*$, and with $R=QQ^t$ as in the statement, we have:
$$\bar{v}=\overline{QuQ^*}=\bar{Q}\bar{u}Q^t=\bar{Q}uQ^t=\bar{Q}Q^*vQQ^t=R^*vR$$ 

Thus $\Gamma_Q$ is presented by the relations $\bar{w}=R^*wR$, with $w=diag(g_1,\ldots,g_N)$. But $(wR)_{ij}=(R\bar{w})_{ij}$ with reads $g_iR_{ij}=R_{ij}g_j^{-1}$, and this gives the result.
\end{proof}

Let us discuss now the group dual case. In the framework of Definition 1.1, we can write as well $A=C^*(\Gamma)$, with $\Gamma$ being a finitely generated discrete quantum group. We have then the Pontrjagin duality formulae $G=\widehat{\Gamma}$ and $\Gamma=\widehat{G}$. See \cite{wo1}.

Following \cite{bbd}, we first have the following result:

\begin{proposition}
Given a discrete group $\Gamma=<g_1,\ldots,g_N>$, consider its dual compact quantum group $G=\widehat{\Gamma}$, diagonally embedded into $U_N^+$. We have then
$$\Gamma_Q=\Gamma/<g_i=g_j|\exists k,Q_{ki}\neq0,Q_{kj}\neq0>$$
with the embedding $\widehat{\Gamma}_Q\subset G=\widehat{\Gamma}$ coming from the quotient map $\Gamma\to\Gamma_Q$.
\end{proposition}

\begin{proof}
Assume indeed that $\Gamma=<g_1,\ldots,g_N>$ is a discrete group, with $\widehat{\Gamma}\subset U_N^+$ coming via $u=diag(g_1,\ldots,g_N)$. With $v=QuQ^*$, we have:
$$\sum_s\bar{Q}_{si}v_{sk}=\sum_{st}\bar{Q}_{si}Q_{st}\bar{Q}_{kt}g_t=\sum_t\delta_{it}\bar{Q}_{kt}g_t=\bar{Q}_{ki}g_i$$

Thus $v_{ij}=0$ for $i\neq j$ gives $\bar{Q}_{ki}v_{kk}=\bar{Q}_{ki}g_i$, which is the same as saying that $Q_{ki}\neq0$ implies $g_i=v_{kk}$. But this latter equality reads $g_i=\sum_j|Q_{kj}|^2g_j$, and we conclude that $Q_{ki}\neq0,Q_{kj}\neq0$ implies $g_i=g_j$, as desired. The converse holds too, see \cite{bbd}.
\end{proof}

More generally, a similar result holds for the arbitrary (non-diagonal) embeddings of the duals of the discrete groups $G=\widehat{\Gamma}$ into $U_N^+$. For details here, we refer to \cite{bbd}.

In what follows we will regard the family $\{\Gamma_Q|Q\in U_N\}$ as a kind of ``bundle'' over the group $U_N$. Observe that, by Proposition 2.3, we cannot expect the correspondence $Q\to\Gamma_Q$ to have any reasonable continuity property, and so our ``bundle'' structure to fit into some known formalism. In addition, we have the following negative result:

\begin{proposition}
When $\Gamma=\widehat{G}$ is a classical group, the fibers $\Gamma_Q$ are trivial (in the sense that they are quotients of $\mathbb Z$), for generic values of $Q\in U_N$.
\end{proposition}

\begin{proof}
This follows indeed from Proposition 2.3. To be more precise, we obtain that $\Gamma_Q$ is trivial, with probability 1, with respect to the Haar measure on $U_N$.
\end{proof}

The above result cuts short any attempt of using probabilistic tools, in order to ``average'' our family of groups $\{\Gamma_Q|Q\in U_N\}$. Would the fibers have been generically non-trival, we could have probably used \cite{csn} in order to average the various numeric invariants of $\Gamma_Q$, in order to obtain a unique, formal ``maximal torus''. But, this is not the case.

Following now \cite{fsk}, we have as well the following result, which can provide counterexamples to some other various naive conjectures which can be made:

\begin{proposition}
For the quantum group of Kac-Paljutkin, and for its generalizations by Sekine, the family $\{\Gamma_Q|Q\in U_N\}$ consists of abelian groups. 
\end{proposition}

\begin{proof}
This follows from the results of Franz and Skalski in \cite{fsk}, who classified all the group dual subgroups of these quantum groups, from \cite{sek}. See \cite{bbd}.
\end{proof}

Finally, we have the following result, regarding Wang's free analogue of the symmetric group \cite{wa2}, which combines various findings from \cite{bbd}, \cite{bi1}, and which is perhaps the most illustrating, for the various phenomena that can appear:

\begin{theorem}
For the quantum permutation group $G=S_N^+$, we have:
\begin{enumerate}
\item Given $Q\in U_N$, the quotient $F_N\to\Gamma_Q$ comes from the following relations:
$$\begin{cases}
g_i=1&{\rm if}\ \sum_lQ_{il}\neq 0\\
g_ig_j=1&{\rm if}\ \sum_lQ_{il}Q_{jl}\neq 0\\ 
g_ig_jg_k=1&{\rm if}\ \sum_lQ_{il}Q_{jl}Q_{kl}\neq 0
\end{cases}$$

\item Given a decomposition $N=N_1+\ldots+N_k$, for the matrix $Q=diag(F_{N_1},\ldots,F_{N_k})$, where $F_N=\frac{1}{\sqrt{N}}(\xi^{ij})_{ij}$ with $\xi=e^{2\pi i/N}$ is the Fourier matrix, we obtain: 
$$\Gamma_Q=\mathbb Z_{N_1}*\ldots*\mathbb Z_{N_k}$$

\item Given an arbitrary matrix $Q\in U_N$, there exists a decomposition $N=N_1+\ldots+N_k$, such that $\Gamma_Q$ appears as quotient of $\mathbb Z_{N_1}*\ldots*\mathbb Z_{N_k}$.
\end{enumerate}
\end{theorem}

\begin{proof}
Here (1) was obtained in \cite{bbd}, via a computation that we will generalize later on, and (2) was obtained as well in \cite{bbd}, via a direct computation. As for (3), the result here goes back to Bichon's work in \cite{bi1}, and can be obtained as well from (1,2). See \cite{bbd}. 
\end{proof}

The above result is quite interesting for us, because it shows that the fibers $\Gamma_Q$ not only wildly vary with $Q$, but are also not subject to quotient maps between them. This phenomenon holds of course as well for $S_N$ itself, but the extension to $S_N^+$ is of interest, because at $N\geq4$ this quantum group is known to be of ``continuous'' nature.

\section{The conjectures}

We present now our series of conjectures, relating the various algebraic and analytic properties of a compact quantum group $G\subset U_N^+$ to those of the family of group duals $\{\widehat{\Gamma}_Q|Q\in U_N\}$, or, equivalently, to those of the family of groups $\{\Gamma_Q|Q\in U_N\}$. 

As already mentioned in the introduction, while the general philosophy for these conjectures is quite old, going back to \cite{bbd}, \cite{bv3}, the statements are new, based on a quite substantial amount of recent work in the area, that we will explain here.

Let us first recall that the characters of the finite dimensional representations of $G$ live in a certain subalgebra $C(G)_{central}\subset C(G)$. To be more precise, $C(G)_{central}$ is by definition the norm closure of the linear span of the characters of irreducible representations, known to be linearly independent. Equivalently, $C(G)_{central}$ is the $C^*$-subalgebra of the tensors in $C(G)\otimes C(G)$ which are symmetric under the comultiplication map $\Delta$ of $C(G)$. In other words, if we denote by $\Sigma$ the tensor flip map of $C(G)\otimes C(G)$, we have:
$$C(G)_{central}=\left\{a\in C(G)\Big|\Delta(a)=\Sigma\circ\Delta(a)\right\}$$

We refer to \cite{lem}, \cite{wo1} for the general theory here.

We recall as well that $G$ is said to be connected if it has no finite quantum group quotient $G\to F\neq\{1\}$. This condition is equivalent to the fact that the coefficient algebra $<r_{ij}>$ must be infinite dimensional, for any nontrivial irreducible unitary representation $r$. In the case of the group duals, $G=\widehat{\Gamma}$, this is the same as asking for $\Gamma$ to have no torsion. For various aspects of the theory here, we refer to \cite{btd},  \cite{cdp}, \cite{pat}, \cite{wa3}.

With this convention, our first conjecture is as follows:

\begin{conjecture}[Characters]
If $G\subset U_N^+$ is connected, for any nonzero $P\in C(G)_{central}$ there exists $Q\in U_N$ such that $\pi_Q:C(G)\to C^*(\Gamma_Q)$ has the property $\pi_Q(P)\neq0$.
\end{conjecture}

Observe that this conjecture holds trivially when $G=\widehat{\Gamma}$ is a group dual, because here we have $\Gamma=\Gamma_Q$, where $Q\in U_N$ to be the spinning matrix which produces the embedding $ \widehat{\Gamma}\subset U_N^+$, coming from Theorem 1.5 above, and we have $\pi_Q=id$ in this case.

The conjecture holds as well in the classical group case, because we can take here $Q\in U_N$ to be such that $QTQ^*\subset\mathbb T^N$, where $T\subset U_N$ is a maximal torus for $G$.

Observe that in both the above cases, we have in fact a matrix $Q\in U_N$ such that $\pi_Q$ is faithful on $C(G)_{central}$. In addition, the connectedness assumption is not really needed in the group dual case, nor for most of the known examples of compact groups. Thus, there are several potential ways of formulating some stronger conjectures.

At the analytic level now, our first, and main conjecture, will concern amenability. Let us recall from \cite{wo1} that associated to any compact quantum group $G$ are in fact several Hopf $C^*$-algebras, including a maximal one $C_{max}(G)$, and a minimal one $C_{min}(G)$. The compact quantum group $G$ is said to be coamenable is the canonical quotient map $C_{max}(G)\to C_{min}(G)$ is an isomorphism. Equivalently, the discrete quantum group $\Gamma=\widehat{G}$ is called amenable when the canonical quotient map $C^*_{max}(G)\to C^*_{min}(G)$ is an isomorphism. With this convention, our conjecture is as follows:

\begin{conjecture}[Amenability]
$G\subset U_N^+$ is coamenable if and only if each of the compact quantum groups $\Gamma_Q$ is coamenable.
\end{conjecture}

Observe that $\implies$ is trivial, because of the quotient map $C(G)\to C^*(\Gamma_Q)$, which can be interpreted as coming from a discrete quantum group quotient map $\widehat{G}\to\Gamma_Q$.

Regarding now $\Longleftarrow$, an equivalent statement here, a bit more convenient, is that if $G$ is not coamenable, then there exists $Q\in U_N$ such that $\Gamma_Q$ is not amenable.

Observe that this latter statement holds trivially in the group dual case, $G=\widehat{\Gamma}$, because we can take here $Q\in U_N$ to be the spinning matrix coming from Theorem 1.5, for which $\pi_Q=id$. The statement holds as well in the classical case, $G\subset U_N$, due to the trivial fact that these latter quantum groups are all coamenable. See \cite{wo1}.

As already mentioned, the above conjecture is our main analytic one. We believe that the above statement is just the ``tip of the iceberg'', with many other conjectures being behind it, some of them regarding the fine analytic structure of $C(G)$, and some other, regarding the fine probabilistic structure of the Kesten measure of $\widehat{G}$.

Regarding now the growth, let us recall from \cite{bv1} that this is constructed by using the balls in $Irr(G)$, with respect to the distance coming from the fundamental corepresentation $u$, and with each corepresentation $r$ contributing with a $\dim(r)^2$ factor.

With this convention, we have the following conjecture:

\begin{conjecture}[Growth]
Assuming $G\subset U_N^+$, the discrete quantum group $\widehat{G}$ has polynomial growth if and only if each $\Gamma_Q$ has polynomial growth.
\end{conjecture}

As before, the conjecture is trivial in the group dual case. In the classical group case the conjecture holds as well, but this is not trivial, coming from  \cite{bv1} in the connected simply connected case, and from the recent paper \cite{dpr} in the general case.

Once again, we believe that this conjecture is just the ``tip of the iceberg''. Here is as series of more specialized statements, regarding the cardinality $|\,.\,|$, the polynomial growth exponents $p(.)$, and the exponential growth exponents $e(.)$:
\begin{eqnarray*}
\log\log|G|&\simeq&\sup_{Q\in U_N}\log\log|\Gamma_Q|\\
p(\widehat{G})&\approx&\sup_{Q\in U_N}p(\Gamma_Q)\\
\log e(\widehat{G})&\approx&\sup_{Q\in U_N}\log e(\Gamma_Q)
\end{eqnarray*}

These statements are all trivial in the group dual case, with equality everywhere. In the classical group case, the first estimate is what seems to come out from the existent literature, and the second estimate is non-trivial, but holds by \cite{dpr}. Finally, regarding the last estimate, this is supported by the various computations in \cite{bv1}.

\section{General results}

We present here some general results, supporting the conjectures made in section 3. As a first statement, collecting the various observations made above, we have:

\begin{proposition}
The $3$ conjectures hold for the classical groups, and for the group duals.
\end{proposition}

\begin{proof}
This follows as explained in the previous section, the summary being;

(1) Characters: trivial for group duals, holds as well for classical groups.

(2) Amenability: trivial for both group duals, and for classical groups.

(3) Growth: trivial for group duals, nontrivial cf. \cite{dpr} for classical groups.
\end{proof}

By getting back now to the precise justifications in section 3, observe that for the group duals, the argument was basically always the same, namely that we can take $Q\in U_N$ to be the spinning matrix coming from Theorem 1.5, for which $\pi_Q=id$. 

Our aim now is to analyse and further extend this phenomenon. The definition that we will need, capturing the fact that ``only one $Q$ is needed'', is as follows:

\begin{definition}
A compact quantum group $G\subset U_N^+$ is called ``tame'' when there exists $L\in U_N$ such that we have a quotient map $\Gamma_L\to\Gamma_Q$, for any $Q\in U_N$.
\end{definition}

Observe that, by changing the fundamental corepresentation, we can always assume that our tame quantum groups are ``normalized'', with $L=1$.

At the level of examples, any group dual is tame. Also, any compact connected Lie group is tame. As a basic counterexample, $S_N$ is not tame, and nor is $S_N^+$.

In the tame case, our conjectures have ``lighter'' formulations, as follows:

\begin{proposition}
When $G$ is tame and normalized, the conjectures are as follows:
\begin{enumerate}
\item Characters: $\pi_1:C(G)\to C^*(\Gamma_1)$ is injective on $C(G)_{central}$.

\item Amenability: $G$ is coamenable if and only if $\Gamma_1$ is amenable.

\item Growth: $\widehat{G}$ has polynomial growth if and only if $\Gamma_1$ has polynomial growth.
\end{enumerate}
\end{proposition}

\begin{proof}
This follows indeed from definitions, by using the quotient maps $\Gamma_1\to\Gamma_Q$ coming from the tameness assumption.
\end{proof}

It would be of course interesting to know more about the tame quantum groups. The problem here is that it is very unclear where the maps $\Gamma_L\to\Gamma_Q$ should come from.

Let us discuss now the verification for Wang's free quantum groups $O_N^+,U_N^+,S_N^+$. We include in our study the hyperoctahedral quantum group $H_N^+$, which, in view of the general theory in \cite{bsp}, is a fundamental example of a free quantum group as well. 

We have the following result, regarding these quantum groups:

\begin{theorem}
The $3$ conjectures hold for $G=O_N^+,U_N^+,S_N^+,H_N^+$.
\end{theorem}

\begin{proof}
We have $3\times4=12$ assertions to be proved, and the idea in each case will be that of using certain special group dual subgroups. We will mostly use the group dual subgroups coming at $Q=1$, which are well-known to be as follows:
$$G=O_N^+,U_N^+,S_N^+,H_N^+\implies\Gamma_1=\mathbb Z_2^{*N},F_N,\{1\},\mathbb Z_2^{*N}$$ 

However, for some of our 12 questions, using these subgroups will not be enough, and we will use as well some carefully chosen subgroups of type $\Gamma_Q$, with $Q\neq1$.

As a last ingredient, we will need some specialized structure results for $G$, in the cases where $G$ is coamenable. Once again, the theory here is well-known, and the situations where $G=O_N^+,U_N^+,S_N^+,H_N^+$ is coamenable, along with the values of $G$, are as follows:
$$\begin{cases}
O_2^+=SU_2^{-1}\\
S_2^+=S_2,S_3^+=S_3,S_4^+=SO_3^{-1}\\
H_2^+=O_2^{-1}
\end{cases}$$

To be more precise, the equalities $S_N^+=S_N$ at $N\leq3$ are known since Wang's paper \cite{wa2}, and the twisting results are all well-known, and we refer here to \cite{bbd}, \cite{byu}.

With these ingredients in hand, we can now go ahead with the proof. It is technically convenient to split the discussion over the 3 conjectures, as follows:

(1) Characters. For $G=O_N^+,U_N^+$, it is known that the algebra $C(G)_{central}$ is polynomial, respectively $*$-polynomial, on the variable $\chi=\sum_iu_{ii}$. Thus, it is enough to show that the variable $\rho=\sum_ig_i$ generates a polynomial, respectively $*$-polynomial algebra, inside the group algebra of the discrete groups $\mathbb Z_2^{*N},F_N$. But for $\mathbb Z_2^{*N}$ this is clear, and by using a multiplication by a unitary free from $\mathbb Z_2^{*N}$, the result holds as well for $F_N$.

Regarding now $G=S_N^+$, we have three cases to be discussed, as follows:

-- At $N=2,3$ this quantum group collapses to the usual permutation group $S_N$, and according to Proposition 4.1 above, the character conjecture holds indeed. 

-- At $N=4$ we have $S_4^+=SO_3^{-1}$, the fusion rules are well-known to be the Clebsch-Gordan ones, and the algebra $C(G)_{central}$ is therefore polynomial on $\chi=\sum_iu_{ii}$. Now observe that Theorem 2.6 gives, with $Q=diag(F_2,F_2)$, the following discrete group: 
$$\Gamma_Q=\mathbb Z_2*\mathbb Z_2=D_\infty$$

Since $Tr(u)=Tr(Q^*uQ)$, the image of $\chi=\sum_iu_{ii}$ in the quotient $C^*(\Gamma_Q)$ is the variable $\rho=2+g+h$, where $g,h$ are the generators of the two copies of $\mathbb Z_2$. Now since this latter variable generates a polynomial algebra, we obtain the result. 

-- At $N\geq5$ the fusion rules are once again known to be the Clebsch-Gordan ones, the algebra $C(G)_{central}$ is, as before, polynomial on $\chi=\sum_iu_{ii}$, and the result follows by functoriality from the result at $N=4$, by using the embedding $S_4^+\subset S_N^+$.

Regarding now $G=H_N^+$, here it is known, from the computations in \cite{bv2}, that the algebra $C(G)_{central}$ is polynomial on the following two variables:
$$\chi=\sum_iu_{ii}\quad,\quad\chi'=\sum_iu_{ii}^2$$

We have two cases to be discussed, as follows:

-- At $N=2$ we have $H_2^+=O_2^{-1}$, and, as explained in \cite{bbd}, with $Q=F_2$ we have $\Gamma_Q=D_\infty$. Let us compute now the images $\rho,\rho'$ of the variables $\chi,\chi'$ in the group algebra of $D_\infty$. As before, from $Tr(u)=Tr(Q^*uQ)$ we obtain $\rho=g+h$, where $g,h$ are the generators of the two copies of $\mathbb Z_2$. Regarding now $\rho'$, let us first recall that the quotient map $C(H_2^+)\to C^*(D_\infty)$ is constructed as follows: 
$$\frac{1}{2}\begin{pmatrix}1&1\\1&-1\end{pmatrix}\begin{pmatrix}u_{11}&u_{12}\\u_{21}&u_{22}\end{pmatrix}\begin{pmatrix}1&1\\1&-1\end{pmatrix}\to\begin{pmatrix}g&0\\0&h\end{pmatrix}$$

Equivalently, this quotient map is constructed as follows:
$$\begin{pmatrix}u_{11}&u_{12}\\u_{21}&u_{22}\end{pmatrix}\to\frac{1}{2}\begin{pmatrix}1&1\\1&-1\end{pmatrix}\begin{pmatrix}g&0\\0&h\end{pmatrix}\begin{pmatrix}1&1\\1&-1\end{pmatrix}=\frac{1}{2}\begin{pmatrix}g+h&g-h\\g-h&g+h\end{pmatrix}$$

We can now compute the image of our character, as follows:
$$\rho'=\frac{1}{2}(g+h)^2=\frac{1}{2}(2+2gh)=1+gh$$

By using now the elementary fact that the variables $\rho=g+h$ and $\rho'=1+gh$ generate a polynomial algebra inside $C^*(D_\infty)$, this gives the result. 

-- Finally, at $N\geq3$ the result follows by functoriality, via the standard diagonal inclusion $H_2^+\subset H_N^+$, from the result at $N=2$, that we established above. 

(2) Amenability. Here the cases where $G$ is not coamenable are those of $O_N^+$ with $N\geq3$, $U_N^+$ with $N\geq2$, $S_N^+$ with $N\geq5$, and $H_N^+$ with $N\geq3$. For $G=O_N^+,H_N^+$ with $N\geq3$ the result is clear, because $\Gamma_1=\mathbb Z_2^{*N}$ is not amenable. Clear as well is the result for $U_N^+$ with $N\geq2$, because $\Gamma_1=F_N$ is not amenable. Finally, for $S_N^+$ with $N\geq5$ the result holds as well, because of the presence of Bichon's group dual subgroup $\widehat{\mathbb Z_2*\mathbb Z_3}$.

(3) Growth. Here the growth is polynomial precisely in the situations where $G$ is infinite and coamenable, the precise cases being $O_2^+=SU_2^{-1}$, $S_4^+=SO_3^{-1}$, $H_2^+=O_2^{-1}$, and the result follows from the fact that the growth invariants are stable by twisting. 
\end{proof}

We can see from the above proof that the verification of the conjectures basically requires to know how to compute $\Gamma_Q$, and to know the representation theory of $G$.

There are many other situations where these two technical ingredients are available, at least to some extent. Without getting into details here, let us just mention that: (1) the product operations $\times,\hat{*}$ can be investigated by using \cite{wa1}, (2) the free complexification operation can be investigated by using \cite{rau}, (3) for deformations, evidence comes from \cite{byu}, \cite{nya}, (4) for free wreath products, evidence comes from \cite{lem}, \cite{lta}, (5) the two-parametric free quantum groups can be studied by using \cite{bs1}, and (6) for the various growth conjectures, substantial evidence comes from the computations in \cite{bv1}, \cite{dpr}.

In short, there is a lot of work to be done. In what follows we will do a part of this work, in relation with two key constructions, coming from Tannakian philosophy. 

\section{Half-liberation}

One interesting discovery coming from Tannakian philosophy is the fact that the commutation relations $ab=ba$ can be succesfully replaced, in relation with several quantum group questions, with the half-commutation relations $abc=cba$. Diagramatically:
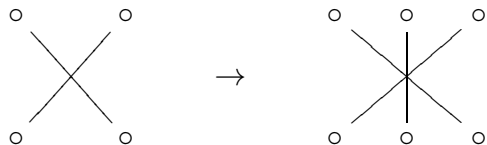
$$\xymatrix@R=12mm@C=10mm{\circ\ar@{-}[dr]&\circ\ar@{-}[dl]\\\circ&\circ}
\qquad\xymatrix@R=5mm@C=5mm{\\ \to\\ }\qquad
\xymatrix@R=12mm@C=5mm{\circ\ar@{-}[drr]&\circ\ar@{-}[d]&\circ\ar@{-}[dll]\\\circ&\circ&\circ}$$

We use here the following notions, coming from \cite{bsp}, \cite{bv3}:

\begin{definition}
The half-classical orthogonal group $O_N^*$ is given by:
$$C(O_N^*)=C(O_N^+)\Big/\Big<abc=cba,\forall a,b,c\in\{u_{ij}\}\Big>$$
The closed quantum subgroups $G\subset O_N^*$ are called half-classical.
\end{definition}

To be more precise now, this definition is motivated by a result from \cite{bv3}, stating that there are exactly 3 categories of pairings, namely those generated by $\emptyset$, $\slash\hskip-2.0mm\backslash$, $\slash\hskip-2.0mm\backslash\hskip-1.69mm|\hskip0.5mm$. Equivalently, there are exactly 3 orthogonal easy quantum groups, namely:
$$O_N\subset O_N^*\subset O_N^+$$

We will be back to these topics in section 6 below. For the moment, we will just use Definition 5.1 as it is, and we refer to \cite{bsp}, \cite{bv3} for where this definition comes from.

We will prove here that the 3 conjectures hold for any half-classical quantum group. In order to do so, we can use the modern approach from \cite{bdu}, which is as follows:

\begin{proposition}
Given a conjugation-stable closed subgroup $H\subset U_N$, consider the algebra $C([H])\subset M_2(C(H))$ generated by the following variables:
$$u_{ij}=\begin{pmatrix}0&v_{ij}\\ \bar{v}_{ij}&0\end{pmatrix}$$
Then $[H]$ is a compact quantum group, we have $[H]\subset O_N^*$, and any non-classical subgroup $G\subset O_N^*$ appears in this way, with $G=O_N^*$ itself appearing from $H=U_N$.
\end{proposition}

\begin{proof}
The $2\times2$ matrices in the statement are self-adjoint, half-commute, and the $N\times N$ matrix $u=(u_{ij})$ that they form is orthogonal, so we have an embedding $[H]\subset O_N^*$. The quantum group property of $[H]$ is also elementary to check, by using an alternative, equivalent construction, with a quantum group embedding as follows:
$$C([H])\subset C(H)\rtimes\mathbb Z_2$$

The surjectivity part is non-trivial, and we refer here to \cite{bdu}.
\end{proof}

We will need as well the following result, also from \cite{bdu}:

\begin{proposition}
We have a bijection $Irr([H])\simeq Irr_0(H)\coprod Irr_1(H)$, where
$$Irr_k(H)=\left\{r\in Irr(H)\Big|\exists l\in\mathbb N,r\in u^{\otimes k}\otimes(u\otimes\bar{u})^{\otimes l}\right\}$$
induced by the canonical identification $Irr(H\rtimes\mathbb Z_2)\simeq Irr(H)\coprod Irr(H)$.
\end{proposition}

\begin{proof}
We have an equality of projective versions $P[H]=PH$, and so an inclusion $Irr_0(H)=Irr(PH)\subset Irr([H])$. The remaining irreducible representations of $[H]$ must come from an inclusion $Irr_1(H)\subset Irr([H])$, appearing as above. See \cite{bdu}.
\end{proof}

Regarding now the maximal tori, the situation is very simple, as follows:

\begin{proposition}
The group dual subgroups $\widehat{[\Gamma]}_Q\subset[H]$ appear via
$$[\Gamma]_Q=[\Gamma_Q]$$
from the group dual subgroups $\widehat{\Gamma}_Q\subset H$ associated to $H\subset U_N$.
\end{proposition}

\begin{proof}
Let us first discuss the case $Q=1$. Consider the diagonal subgroup $\widehat{\Gamma}_1\subset H$, with the associated quotient map $C(H)\to C(\widehat{\Gamma}_1)$ denoted $v_{ij}\to\delta_{ij}h_i$. At the level of the algebras of $2\times2$ matrices, this map induces a quotient map $M_2(C(H))\to M_2(C(\widehat{\Gamma}_1))$, and our claim is that we have a factorization, as follows:
$$\begin{matrix}
C([H])&\subset&M_2(C(H))\\
\\
\downarrow&&\downarrow\\
\\
C([\widehat{\Gamma}_1])&\subset&M_2(C(\widehat{\Gamma}_1))
\end{matrix}$$

Indeed, it is enough to show that the standard generators of $C([H])$ and of $ C([\widehat{\Gamma}_1])$ map to the same elements of $M_2(C(\widehat{\Gamma}_1))$. But these generators map indeed as follows:
$$\begin{matrix}
u_{ij}&\to&\begin{pmatrix}0&v_{ij}\\ \bar{v}_{ij}&0\end{pmatrix}\\
\\
&&\downarrow\\
\\
\delta_{ij}v_{ij}&\to&\begin{pmatrix}0&\delta_{ij}h_i\\ \delta_{ij}h_i^{-1}&0\end{pmatrix}
\end{matrix}$$

Thus we have the above factorization, and since the map on the left is obtained by imposing the relations $u_{ij}=0$ with $i\neq j$, we obtain $[\Gamma]_1=[\Gamma_1]$, as desired.

In the general case now, $Q\in U_N$, the result follows by applying the above $Q=1$ result to the quantum group $[H]$, with fundamental corepresentation $w=QuQ^*$.
\end{proof}

Now back to our conjectures, we have the following result:

\begin{theorem}
The $3$ conjectures hold for any half-classical quantum group of the form $[H]\subset O_N^*$, with $H\subset U_N$ being connected.
\end{theorem}

\begin{proof}
By using Proposition 4.1, we know that the conjectures hold for $H\subset U_N$. The idea will be that of ``transporting'' these results, via $H\to [H]$:

(1) Characters. We can pick here a maximal torus $T=\Gamma_Q$ for the compact group $H\subset U_N$, and by using the formula $[\Gamma]_Q=[\Gamma_Q]=[T]$ from Proposition 5.4 above, we obtain the result, via the identification in Proposition 5.3.

(2) Amenability. There is nothing to be proved here, because $O_N^*$ is coamenable, and so are all its quantum subgroups. Note however, in relation with the comments made in section 3 above, that in the connected case, the Kesten measures of $G,[T]$ are intimately related. For some explicit formulae here, for $G=O_N^*$ itself, see \cite{bsp}.

(3) Growth. Here the situation is similar to the one for the amenability conjecture, because by Proposition 5.3 above, $[H]$ has polynomial growth.
\end{proof}

The above result is waiting for a number of extensions. First, we believe that the connectivity assumption on $H$ should simply follow from the connectivity of $[H]$, and so, that this assumption can be dropped. With this result in hand, and by using as well Proposition 4.1 and Proposition 5.2, we would have the conjectures for any $G\subset O_N^*$.

Interesting as well would be to have a full extension of the classical results, with a statement covering all the closed subgroups $G\subset U_N$. This looks possible, some general complex half-liberation theory being available from \cite{bbi}.

\section{Tannakian aspects}

In this section we present a systematic Tannakian approach to our various conjectures. Our starting point is the following result, coming from Woronowicz's work in \cite{wo2}:

\begin{proposition}
Given an inclusion $G\subset O_N^+$, with the corresponding fundamental corepresentations denoted $u\to w$, we have the following formula:
$$C(G)=C(O_N^+)\Big/\Big(T\in Hom(u^{\otimes k},u^{\otimes l}),\forall k,l\in\mathbb N,\forall T\in Hom(w^{\otimes k},w^{\otimes l})\Big)$$
A similar result holds in the unitary case, by assuming that $k,l$ are ``colored'' integers, with the tensor powers $v^{\otimes k},v^{\otimes l}$ being obtained by tensoring $v,\bar{v}$.
\end{proposition}

\begin{proof}
This follows indeed from \cite{wo2}. For a short, recent proof, see \cite{mal}.
\end{proof}

Regarding now the tori, at this level of generality, we have the following result:

\begin{proposition}
The intertwining formula $T\in Hom(u^{\otimes k},u^{\otimes l})$, with $u=QvQ^*$, where $v=diag(g_1,\ldots,g_N)$, is equivalent to the collection of conditions
$$(T^Q)_{j_1\ldots j_l,i_1\ldots i_k}\neq0\implies g_{i_1}\ldots g_{i_k}=g_{j_1}\ldots g_{j_l}$$
one for each choice of the multi-indices $i,j$, where $T^Q=(Q^*)^{\otimes l}TQ^{\otimes k}$. 
\end{proposition}

\begin{proof}
Observe first that, by conjugating by $Q$, we have the following formula:
$$T\in Hom(u^{\otimes k},u^{\otimes l})\iff T^Q\in Hom(v^{\otimes k},v^{\otimes l})$$

Thus, it is enough to prove the result at $Q=1$. And here, with standard multi-index notations, including the convention $g_i=g_{i_1}\ldots g_{i_k}$, the computation goes as follows:
\begin{eqnarray*}
T\in Hom(u^{\otimes k},u^{\otimes l})
&\iff&Tu^{\otimes k}e_i=u^{\otimes l}Te_i,\forall i\\
&\iff&Te_i\otimes g_i=u^{\otimes l}\sum_jT_{ji}e_j,\forall i\\
&\iff&\sum_jT_{ji}e_j\otimes g_i=\sum_jT_{ji}e_j\otimes g_j,\forall i\\
&\iff&T_{ji}g_i=T_{ji}g_j,\forall i,j\\
&\iff&[T_{ji}\neq0\implies g_i=g_j],\forall i,j
\end{eqnarray*}

Thus we have obtained the relation in the statement, and we are done.
\end{proof}

In principle Propositions 6.1 and 6.2 give all the needed ingredients for a Tannakian approach to our conjectures. Obviously, there is a lot of work to be done here.

Let us discuss now the easy case, where more concrete results can be obtained. We use the framework of \cite{twe}. Let $P(k,l)$ be the set of partitions between an upper row of $k$ points, and a lower row of $l$ points, with each leg colored black or white, and with $k,l$ standing for the corresponding ``colored integers''. We have then:

\begin{definition}
A category of partitions is a collection of sets $D=\bigcup_{kl}D(k,l)$, with $D(k,l)\subset P(k,l)$, which contains the identity, and is stable under:
\begin{enumerate}
\item The horizontal concatenation operation $\otimes$.

\item The vertical concatenation $\circ$, after deleting closed strings in the middle.

\item The upside-down turning operation $*$ (with reversing of the colors).
\end{enumerate}
\end{definition}

As explained in \cite{mal}, \cite{twe}, such categories produce quantum groups. To be more precise, associated to any partition $\pi\in P(k,l)$ is the following linear map:
$$T_\pi(e_{i_1}\otimes\ldots\otimes e_{i_k})=\sum_{j:\ker(^i_j)\leq\pi}e_{j_1}\otimes\ldots\otimes e_{j_l}$$

Here the kernel of a multi-index $(^i_j)=(^{i_1\ldots i_k}_{j_1\ldots j_l})$ is the partition obtained by joining the sets of equal indices. With this construction in hand, we have:

\begin{definition}
A compact quantum group $G\subset U_N^+$ is called easy when
$$Hom(u^{\otimes k},u^{\otimes l})=span\left(T_\pi\Big|\pi\in D(k,l)\right)$$
for any $k,l$, for a certain category of partitions $D\subset P$.
\end{definition}

In other words, the easiness condition states that the Schur-Weyl dual of $G$ comes in the simplest possible way: from partitions. As a basic example, according to an old result of Brauer \cite{bra}, the group $G=U_N$ is easy, with $D=P_2$ being the category of color-matching pairings. Easy as well is $U_N^+$, with $D=NC_2\subset P_2$ being the category of noncrossing color-matching pairings. See \cite{bsp}, \cite{bv3}, \cite{fre}, \cite{mal}, \cite{rw2}, \cite{twe}.

With these conventions, we have the following result:

\begin{theorem}
In the uncolored case, the intertwining formula $T_\pi\in Hom(u^{\otimes k},u^{\otimes l})$, with $u=QvQ^*$, where $v=diag(g_1,\ldots,g_N)$, is equivalent to
$$\delta_\pi^Q\binom{i_1\ldots i_k}{j_1\ldots j_l}\neq 0\implies g_{i_1}\ldots g_{i_k}=g_{j_1}\ldots g_{j_l}$$
with the generalized Kronecker symbols being given by $\delta_\pi^Q=\prod_{\beta\in\pi}\delta_\beta^Q$, with:
$$\delta_{1^r_p}^Q\binom{i_1\ldots i_r}{j_1\ldots j_p}=\sum_sQ_{si_1}\ldots Q_{si_r}\bar{Q}_{sj_1}\ldots\bar{Q}_{sj_p}$$
A similar result holds is the colored case, with the convention $g_\bullet=g^{-1}$.
\end{theorem}

\begin{proof}
With multi-index notations, as in the proof of Proposition 6.2, we have:
\begin{eqnarray*}
T_\pi^Q(e_i)
&=&(Q^*)^{\otimes l}T_\pi Q^{\otimes k}e_i\\
&=&\sum_s(Q^*)^{\otimes l}T_\pi(Q^{\otimes k})_{si}e_s\\
&=&\sum_s(Q^*)^{\otimes l}(Q^{\otimes k})_{si}\sum_t\delta_\pi\binom{s}{t}e_t\\
&=&\sum_{stj}\delta_\pi\binom{s}{t}(Q^{\otimes k})_{si}((Q^*)^{\otimes l})_{jt}e_j
\end{eqnarray*}

Thus, with full indices now, we have the following formula:
$$(T_\pi^Q)_{j_1\ldots j_l,i_1\ldots i_k}=\sum_{s_1\ldots s_k}\sum_{t_1\ldots t_l}\delta_\pi\binom{s_1\ldots s_k}{t_1\ldots t_l}Q_{s_1i_1}\ldots Q_{s_ki_k}\bar{Q}_{t_1j_1}\ldots\bar{Q}_{t_lj_l}$$

Since this quantity is multiplicative with respect to the blocks of $\pi$, by decomposing over these blocks, we obtain the formula in the statement.
\end{proof}

Observe that the above result generalizes the computation in Theorem 2.6. In view of the similarities with Proposition 2.3, one interesting question, that we would like to raise here, is that of extending Theorem 6.5, as to cover Proposition 2.3 as well.

At a more concrete level now, the orthogonal easy quantum groups were classified by Raum and Weber in \cite{rw2}. Without getting into details here, let us mention that:
\begin{enumerate}
\item The classical examples are covered by Proposition 4.1. The free examples consist of the quantum groups $S_N^+,H_N^+,O_N^+$ from Theorem 4.4, and then of $S_N'^+,B_N^+$ and $B_N'^+,B_N''^+$, which appear as free versions of $S_N'=S_N\times\mathbb Z_2,B_N\simeq O_{N-1}$ and of $B_N'\simeq O_{N-1}\times\mathbb Z_2$, taken twice, where the methods for $S_N^+,O_N^+$ apply. 

\item Then, we have a number of half-liberations, covered by Theorem 5.5, and an uncountable family, constructed in \cite{rw1}. For this latter family we can use the diagonal group dual subgroup $\widehat{\Gamma}_1$, and by using the crossed product picture in \cite{rw1} we conclude that our various conjectures hold indeed.

\item Finally, we have a last series, constructed in \cite{rw2}. Here the quantum groups are not coamenable, and nor are their diagonal group dual subgroups $\widehat{\Gamma}_1$, so the amenability and growth conjectures are both satisfied. The remaining problem regards the character conjecture, and we have no results here.
\end{enumerate}

In the easy unitary case, where the classification so far is only available in the classical and free cases \cite{twe}, the situation is quite unclear. Regarding amenability, an idea here would be that of trying to relate the Kesten coamenability of $G$ to the random walk on the various groups $\Gamma_Q$, with the restrictions on these latter random walks coming from the formula in Theorem 6.5. However, this looks like a quite technical task.

Finally, we believe that the present conjectures, and the maximal torus philosophy in general, can be of help in connection with certain rigidity questions in noncommutative geometry. As mentioned in section 1, the motivating remark from \cite{bbd} was the fact that a non-classical group dual cannot act on a compact connected Riemannian manifold. Thus, if a quantum group $G\subset U_N^+$ has the property that at least one of its subgroups $\widehat{\Gamma}_Q$ is non-classical, then $G$ cannot act on a compact connected Riemannian manifold.

In view of \cite{gjo}, where the rigidity conjecture was proved in the general case, for any non-classical subgroup $G\subset U_N^+$, all this is obsolete. However, and here comes our point, we believe that the original philosophy in \cite{bbd} can applied to a wider range of rigidity questions, where some of the tools from \cite{gjo} might not be available. For a study of the half-classical case, based on \cite{bs2}, \cite{bi2}, we refer to the recent article \cite{ban}.

\end{document}